\documentclass[a4paper] {article}
[12pt]

\makeatletter
\renewcommand*\l@section{\@dottedtocline{1}{1.5em}{2.3em}}
\makeatother

\usepackage{amsfonts}
\usepackage{amssymb}
\usepackage[T1]{fontenc}

\usepackage{tikz}
\usetikzlibrary{calc}

\usepackage{CJK}
\usepackage{amsmath}
 
\usepackage{amsfonts}
\usepackage{amssymb}
\usepackage{amsthm}
\usepackage{amssymb}
\usepackage{enumerate}
\usepackage[calc]{picture}
\usepackage[all,cmtip]{xy}

\usepackage[mathscr]{eucal}
\usepackage{eqlist}

\usepackage{color}
\usepackage{abstract} 
\usepackage[T1]{fontenc}
 
\usepackage[top=2.8cm, bottom=2.8cm, left=2.5cm, right= 2.5cm]{geometry}

\theoremstyle{plain}
\newtheorem{theorem}{Theorem}
\newtheorem{proposition}[theorem]{Proposition}
\newtheorem{lemma}[theorem]{Lemma}
\newtheorem{Proposition}[theorem]{Proposition}

\newtheorem{example}[theorem]{Example}
\newtheorem{corollary}[theorem]{Corollary}

\theoremstyle{definition}
\newtheorem{definition}{Definition}

\usepackage{etoolbox}
\newtheoremstyle{myrem}
 {3pt}
 {3pt}
 {\normalsize}
 { }
 {\itshape}
 {:}
 { }
 {}

 \theoremstyle{myrem}
 \newtheorem{remark}{Remark}
 \appto\remark{\leftskip\parindent}
 \appto\remark{\rightskip\parindent}

\numberwithin{equation}{section}
\numberwithin{theorem}{section}

\begin{document}
\title{}
 \author{Shiquan Ren}

\begin{center}
{\Large {\textbf {The Embedded  Homology  of Hypergraphs  and  Applications
}}}
 \vspace{0.58cm}

Stephane Bressan,  Jingyan Li,   Shiquan Ren,          Jie Wu 

\end{center}

{ 
\begin{quote}
\begin{abstract}
\medskip
Hypergraphs are   mathematical models for many problems in data sciences.  In recent decades, the topological properties of hypergraphs have been studied and  various kinds of (co)homologies   have been  constructed (cf. \cite{hg,betti,parks}).  In this paper,  generalising the usual homology of simplicial complexes, we define the embedded  homology  of hypergraphs   as well as the persistent embedded  homology  of sequences of hypergraphs.   As a generalisation of the Mayer-Vietoris sequence for the homology of simplicial complexes, we give a Mayer-Vietoris  sequence for the embedded homology of hypergraphs.  Moreover,  as applications of the   embedded homology,  we study  acyclic hypergraphs and  construct some indices for the data analysis of hyper-networks. 
   
 \end{abstract}

\end{quote}
}
 


\footnotetext[1]
{{\bf{AMS Mathematical Classifications 2010}}.  	 Primary  55U10, 	55U15;     Secondary 68P05,  68P15.}

 



\section{Introduction}\label{sec9.0}


Given a set $V$, we define its   {\bf  power set}  $\Delta[V]$ as the collection of all the non-empty subsets of $V$.  Throughout this paper, we assume that   $V$ is a finite and totally ordered set.  For any subset $\{v_0,\cdots,v_n\}$ of $V$, we assume $v_0\prec \cdots \prec v_n$ if there is no extra claim, where $\prec$ is the total order of $V$.   
  
In topology, simplicial complexes are   classical models  for triangulated topological spaces.  An {\bf (abstract) simplicial complex} $K$ is a pair $(V_K,K)$ where $V_K$ is a set and $K$ is a subset of $\Delta[V_K]$  satisfying the following conditions (cf. \cite[p. 107]{hatcher}):

(i).  for any $v\in V_K$, the single-point set $\{v\}$ is in  $K$; 

(ii).  for any $\sigma\in K$ and any non-empty subset $\tau\subseteq \sigma$,     $\tau$ is in  $K$. 

\noindent  The elements of $V_K$ are called {\bf vertices}, and the elements of $K$ are called {\bf simplices}.  A simplex   consisting of $n+1$ elements in $V$, $n\geq 0$, is called an {\bf $n$-simplex}. By (i), we can identify the $0$-simplices of $K$  with the  {vertices}.  Hence 
  the simplicial complex $(V_K,K)$ can be simply denoted as $K$.  For any $n\geq 0$,  the collection of all $n$-simplices in $K$ are denoted as $K_n$.      The {\bf dimension} of a simplicial complex $K$ is the largest integer $n$ such that $K_n$ is non-empty.  For any $n\geq 1$, a {\bf$(n-1)$-face} of an $n$-simplex is an $(n-1)$-simplex  obtained by removing one vertex in the $n$-simplex.

Let $G$ be an abelian group.  Given a non-empty  finite set $S$, we use $G(S)$ to denote the collection of all formal linear combinations of the elements in $S$ with coefficients in $G$. In particular,     $\mathbb{Z}(S)$ is the free $\mathbb{Z}$-module generated by $S$.

\begin{example}\cite[pp. 103 - 105]{hatcher}\label{ex9.1.1}
A {\bf\it standard $n$-simplex}   is denoted as 
\begin{eqnarray*}
\Delta^n=\{v_0,\cdots,v_n\}. 
\end{eqnarray*}
The $(n-1)$-faces of the standard $n$-simplex $\Delta^n$ are denoted as 
\begin{eqnarray*}
 \Delta^{n-1}_i=\{v_0,\cdots,\hat{v}_i,\cdots,v_n\},  \text{\ \ \ } 0\leq i\leq n.
 \end{eqnarray*} 
We have the face maps  $d_i$ sending   $\Delta^n$ to  $\Delta^{n-1}_i$, for $0\leq i\leq n$.  And we have the boundary maps 
\begin{eqnarray*}
\partial_n: G(\Delta^n)\longrightarrow G(\Delta^{n-1}_0,\cdots, \Delta^{n-1}_n)
 \end{eqnarray*}
 given by $\partial_n=\sum_{i=0}^n (-1)^i d_i$,  which extends linearly over $G$.  The power set of $\Delta^n$ is denoted as $\Delta[n]$,  called {\bf\it the  standard simplicial complex spanned by $n+1$ vertices}. 
\end{example}
 
 \smallskip

Hypergraph was invented as a model for hyper-networks by data scientists. Mathematically,  the hypergraph  is a  generalisation of the notion of the simplicial complex.   A {\bf hypergraph}  is a pair $(V_\mathcal{H},\mathcal{H})$ where $V_\mathcal{H}$ is a set  and $\mathcal{H}$ is a subset of  $\Delta[V_\mathcal{H}]$  (cf. \cite{berge,parks}).  An element of $V_\mathcal{H}$ is called a {\bf vertex} and an element of $\mathcal{H}$ is called a {\bf hyperedge}. For any $n\geq 0$, we call a hyperedge consisting of $n+1$ vertices   an {\bf $n$-hyperedge}, and we  denote $\mathcal{H}_n$ as the collection of  all the $n$-hyperedges of $\mathcal{H}$. We define the {\bf dimension} of  $\mathcal{H}$ as the largest integer $n$ such that $\mathcal{H}_n$ is non-empty.  
In this paper, we assume that $V_\mathcal{H}$ is the union of all the vertices of the hyperedges of $\mathcal{H}$ and we simply denote the hypergraph $(V_\mathcal{H},\mathcal{H})$ as $\mathcal{H}$.   

 Given a hypergraph $\mathcal{H}$, if for any hyperedge $\sigma\in\mathcal{H}$,   any non-empty subsets of  $\sigma$ are hyperedges of $\mathcal{H}$, then $\mathcal{H}$ is a simplicial complex,  and the hyperedges of $\mathcal{H}$ are  simplices.

\smallskip

In recent decades, various   (co)homology  theories of hypergraphs have been intensively studied. In 1991, A.D. Parks and S.L. Lipscomb \cite{parks}  defined the associated simplicial complex of a hypergraph, that is, the minimal simplicial complex that a hypergraph can be embedded in.  They also studied the  homology  of the associated simplicial complex.  In 1992,  F.R.K. Chung and R.L. Graham \cite{hg} constructed certain cohomology for hypergraphs, with mod $2$ coefficients, in a combinatorial way. In 2009, E. Emtander \cite{betti} constructed certain simplicial complexes  for hypergraphs (called the independence simplicial complexes) and studied the homology of these simplicial complexes;  and J. Johnson   \cite{com1}  applied the topology of hypergraphs  to study hyper-networks of complex systems.  

A graph is a hypergraph whose hyperedges consist of at most two vertices. And a directed graph, or simply called a digraph,  is the geometric object obtained   by associating a direction  with each edge of a graph.    Since 2012,  the  homology  theory of graphs and digraphs has attracted the attention of A. Grigor'yan, Y. Lin, Y. Muranov and S.T. Yau  \cite{yau3,yau1,yau4,yau2,yau5}.

\smallskip

In this paper, by generalising \cite{yau1} and using the associated simplicial complexes defined in \cite{parks}, we construct the embedded  homology  of hypergraphs and study the persistent embedded  homology  of sequences of hypergraphs.  In particular, if the hypergraph is a simplicial complex, then the embedded homology coincides with the usual homology.  Moreover, generalising the Mayer-Vietoris sequence for the homology of simplicial complexes (cf. \cite[pp. 149 - 153]{hatcher}),  we give a Mayer-Vietoris sequence for the embedded homology as  well as a persistent version of Mayer-Vietoris sequence  for the persistent embedded homology of hypergraphs.  Furthermore, we apply  the associated simplicial complexes and the embedded homology to characterise an important family of hypergraphs:  the acyclic hypergraphs.  Finally, as applications of the embedded homology of hypergraphs in data sciences, we construct some topological indices for the data analysis of hyper-networks. 

The paper is organised as follows. In Section~\ref{sec88}, by generalising the definition of the homology of (directed) graphs given by A. Grigor'yan, Y. Lin, Y. Muranov and S.T. Yau \cite[Section~3.3]{yau1},    we construct certain  homology  for graded groups embedded in chain complexes.  In Section~\ref{sec9.2},  we define  the embedded  homology  of hypergraphs using the associated simplicial complexes defined in \cite{parks}. Moreover,  under certain conditions, we give a Mayer-Vietoris sequence for the embedded  homology  of hypergraphs in Theorem~\ref{c9.3.3.1}.   In Section~\ref{sec9.7},  we study the persistent embedded  homology  of sequences hypergraphs and give a persistence version of the Mayer-Vietoris sequence  in Theorem~\ref{c9.4.8}.  
In Section~\ref{sec9.5}, we use the associated simplicial complexes and the embedded homology to study   acyclic hypergraphs. We give some  characterisations for the associated simplicial complexes of  acyclic hypergraphs in Theorem~\ref{th9.5.2}.  And we study the embedded homology of a particular family of acyclic hypergraphs in Subsection~\ref{subsec9.5.2}.  In Section~\ref{sec9.6}, we apply the embedded homology of hypergraphs and construct some indices to measure the connectivity of hyper-networks, the differentiation of hyper-networks with respect to a function on the vertices, and the correlation of two functions on the vertices of hyper-networks.  



 \section{Homology  of Graded Groups Embedded in  Chain Complexes}\label{sec88}

In this section, we construct  the  homology  of  graded groups which are embedded in chain complexes.  

\smallskip

Firstly, we review the homology of  chain complexes.  Let  $C_n$, $n=0,1,2,\cdots,$ be a sequence of abelian groups such that there exists a sequence of  homomorphisms  (called {\bf boundary maps})
\begin{eqnarray}\label{chain1}
\cdots \overset{\partial_{n+2}}{\longrightarrow}C_{n+1}\overset{\partial_{n+1}}{\longrightarrow} C_n \overset{\partial_{n}}{\longrightarrow}C_{n-1}\overset{\partial_{n-1}}{\longrightarrow}\cdots \overset{\partial_{2}}{\longrightarrow} C_1 \overset{\partial_{1}}{\longrightarrow}C_0\overset{\partial_{0}}{\longrightarrow}0 
\end{eqnarray}
with $\partial_n\circ \partial_{n+1}=0$ for each $n$.  Such a  sequence  (\ref{chain1}) of abelian groups and homomorphisms  
 is called a {\bf chain complex} (cf. \cite[p. 106]{hatcher}).  Both the intersection and the direct sum of a family of chain complexes  are still chain complexes.  The {\bf $n$-th homology}  of the chain complex (\ref{chain1}) is defined as the quotient group\footnotemark[3] (cf. \cite[p. 106]{hatcher})\footnotetext[3]{We use $(\cdot)_*$ to denote a sequence of objects, where $*=0,1,2,\cdots$.	 } 
 \begin{eqnarray*}
H_n(\{C_*, \partial_*\})=\text{Ker}\partial_n/\text{Im}\partial_{n+1}. 
\end{eqnarray*}

 \smallskip

Secondly, we consider graded groups embedded in chain complexes and define the infimum chain complexes and the supremum chain complexes.  For each $n\geq 0$, let $D_n$ be a subgroup of $C_n$. In particular,  if for each $n\geq 1$, $\partial D_n\subseteq D_{n-1}$, then we call the sequence $D_*$ a {\bf subchain complex} of $C_*$.    Given a subchain complex of a chain complex, the sequence of quotient groups, equipped with the respective quotient maps,  is still a chain complex.

\begin{definition}\label{def9.2.2}
Given a graded group $D_*$ embedded in a chain complex $C_*$, the {\bf infimum chain complex} $\text{Inf}_*(D_*,C_*)$ of the sequence $\{D_*,C_*\}$ is the chain complex
\begin{eqnarray*}
\text{Inf}_n(D_*,C_*)=\sum \{C_n'\mid C_*' \text{ is a subchain complex of }  C_*  \text{ and }  C'_n\subseteq D_n\}.
\end{eqnarray*}
\end{definition}
It follows immediately from Definition~\ref{def9.2.2} that if $D_*$ is a sub-chain complex, then 
$\text{Inf}_*(D_*,C_*)=D_*$. 

In \cite[Section~3.3]{yau1},  A. Grigor'yan, Y. Lin, Y. Muranov and S.T. Yau constructed the chain complex 
$D_n\cap \partial_n^{-1}(D_{n-1})$ where $\partial_n^{-1}$ denotes the pre-image of $\partial_n$. 
In the next proposition, we show that
our infimum chain complex in Definition~\ref{def9.2.2} coincides with the explicit formula given in \cite[Section~3.3]{yau1}.

\begin{proposition}\label{p9.2.88}
Given a chain complex $C_*$ and a sequence of subgroups $D_*$, the infimum chain complex is given by
\begin{eqnarray}\label{eq9.2.6}
\text{Inf}_n(D_*,C_*)=D_n\cap \partial_n^{-1}(D_{n-1}).   
\end{eqnarray}
\end{proposition} 
\begin{proof}
In order to prove (\ref{eq9.2.6}), we need to show that $\{D_n\cap \partial_n^{-1}(D_{n-1})\}_{n\geq 0}$ is the largest subchain complex of $C_*$ that are contained in $D_*$ as a  graded abelian group.  Let $\alpha\in D_n\cap \partial_n^{-1}(D_{n-1})$.  Then $\alpha\in D_n$ and $\partial_n\alpha\in D_{n-1}$. Since $\partial_{n-1}\partial_n\alpha=0$, 
\begin{eqnarray*}
\partial_n\alpha\in D_{n-1}\cap \partial^{-1}_{n-1}(D_{n-2}). 
\end{eqnarray*}
Thus $\{D_n\cap \partial_n^{-1}(D_{n-1})\}_{n\geq 0}$ is a subchain complex of $C_*$.  Let $C'_*\subseteq D_*$ be a subchain complex of $C_*$.  Let $\alpha' \in C'_n$.  Then $\alpha'\in D_n$ and $\partial_n\alpha'\in D_{n-1}$. Hence 
\begin{eqnarray*}
\alpha' \in D_n\cap \partial_n^{-1}(D_{n-1}). 
\end{eqnarray*}
Thus for each $n\geq 0$, 
\begin{eqnarray*}
C'_n\subseteq D_n\cap \partial_n^{-1}(D_{n-1}). 
\end{eqnarray*}
Therefore, $\{D_n\cap \partial_n^{-1}(D_{n-1})\}_{n\geq 0}$ is the largest subchain complex of $C_*$ that are contained in $D_*$ as a graded abelian group. 
\end{proof}
\begin{remark}
In (\ref{eq9.2.6}), when $n=0$, we let $\partial_0=0$ and $D_{-1}=0$.  
\end{remark}

\begin{definition}\label{def-nov26}
Given a graded group $D_*$ embedded in a chain complex $C_*$, the {\bf supremum chain complex} $\text{Sup}_*(D_*,C_*)$ of the sequence $\{D_*,C_*\}$ is the chain complex
\begin{eqnarray*}
\text{Sup}_n(D_*,C_*)=\bigcap \{C_n'\mid C_*' \text{ is a subchain complex of }  C_* 
 \text{ and  }  D_n\subseteq C_n'\}.
\end{eqnarray*}
\end{definition}
It follows immediately from Definition~\ref{def-nov26} that if $D_*$ is a sub-chain complex, then 
 $\text{Sup}_n(D_*,C_*)=D_*$.


\begin{proposition}\label{p9.2.888}
Given a chain complex $C_*$ and a sequence of subgroups $D_*$,  the supremum chain complex is given by
\begin{eqnarray}\label{eq9.2.6-s}
\text{Sup}_n(D_*,C_*)=D_n + \partial_{n+1}D_{n+1}.
\end{eqnarray}
\end{proposition} 
\begin{proof}
In order to prove (\ref{eq9.2.6-s}), we need to show that $\{D_n + \partial_{n+1}D_{n+1}\}_{n\geq 0}$ is the smallest subchain complex of $C_*$ that contains $D_*$ as a graded abelian group. Let $\bar\alpha\in D_n+ \partial_{n+1} D_{n+1}$.  Then 
\begin{eqnarray*}
\bar\alpha=\beta_n+\partial_{n+1}\beta_{n+1}
\end{eqnarray*}
  where $\beta_*\in D_*$.  Since $\partial_n\partial_{n+1} \beta_{n+1}=0$,  we have
$\partial_n\bar\alpha=\partial_n\beta_n$.  
Hence 
\begin{eqnarray*}
\partial_n\bar\alpha\in D_{n-1}+ \partial_n D_n. 
\end{eqnarray*}
Hence $\{D_n + \partial_{n+1}D_{n+1}\}_{n\geq 0}$ is a subchain complex of $C_*$.  Let $C'_*\supseteq D_*$ be a subchain complex of $C_*$. Then $\beta_n\in C'_n$ and $\partial_{n+1}\beta_{n+1}\in C'_{n}$.  Thus $\bar \alpha\in C'_n$. Thus for each $n\geq 0$, 
\begin{eqnarray*}
  D_n + \partial_{n+1}D_{n+1}\subseteq C'_n.  
\end{eqnarray*}
Therefore, $\{D_n + \partial_{n+1}D_{n+1}\}_{n\geq 0}$ is the smallest subchain complex of $C_*$ that contains $D_*$ as a graded abelian group. 
\end{proof}



Given two graded subgroups $D_*$ and $D_*'$, by a straight-forward computation, we have some basic properties of the infimum chain complexes and the supremum chain complexes, in next proposition. 
 \begin{proposition}\label{pr9.2.18}
Let $C_*$ be a chain complex and $D_*,D_*'$ be graded subgroups of $C_*$.  Then 
\begin{eqnarray*}
\text{Inf}_n (D_*\cap D_*',C_*)&=&\text{Inf}_n(D_*,C_*)\cap \text{Inf}_n(D_*',C_*),\\
\text{Inf}_n (D_*+ D_*',C_*)&\supseteq&\text{Inf}_n (D_*,C_*)+\text{Inf}_n ( D_*',C_*), \\
\text{Sup}_n (D_*\cap D_*',C_*)&\subseteq&\text{Sup}_n(D_*,C_*)\cap \text{Sup}_n(D_*',C_*),\\
\text{Sup}_n (D_*+ D_*',C_*)&=&\text{Sup}_n(D_*,C_*)+ \text{Sup}_n(D_*',C_*).
\end{eqnarray*}
 \end{proposition}

\smallskip

Thirdly, we study the homology of the infimum chain complexes and the  supremum chain complexes.  It follows from Proposition~\ref{p9.2.88} and \cite[Proposition~3.13]{yau1} that
 \begin{eqnarray*}
&& H_n(\text{Inf}_*(D_*,C_*))\\
 &=&\text{Ker}(\partial_n|_{D_n\cap \partial_n^{-1}(D_{n-1})})/\text{Im}(\partial_{n+1}|_{D_{n+1}\cap \partial_{n+1}^{-1}(D_{n})})\\
 &=&\text{Ker}(\partial_n|_{D_n})/(D_n\cap \partial_{n+1}D_{n+1}). 
 \end{eqnarray*}
 Moreover,  we have the next proposition.
\begin{proposition}\label{prop9.2.8}
The homology of $\text{Inf}_*(D_*,C_*)$ and the  homology of $\text{Sup}_*(D_*,C_*)$ are isomorphic. 
\end{proposition}
\begin{proof}
With the help of the isomorphism theorem of groups,  we have 
\begin{eqnarray*}
&& H_n(\text{Sup}_*(D_*,C_*))\\
 &=&\text{Ker}(\partial_n|_{D_n+\partial_{n+1}D_{n+1}})/\text{Im}(\partial_{n+1}|_{D_{n+1}+\partial_{n+2} D_{n+2}})\\
 &=&(\partial_{n+1}D_{n+1}+\text{Ker}(\partial_n|_{D_n}))/\partial_{n+1} D_{n+1}\\
 &\cong&\text{Ker}(\partial_n|_{D_n})/\text{Ker}(\partial_n|_{D_n})\cap \partial_{n+1} D_{n+1}\\
 &=&\text{Ker}(\partial_n|_{D_n})/(D_n\cap \partial_{n+1}D_{n+1})\\
 &=& H_n(\text{Inf}_*(D_*,C_*)).     
\end{eqnarray*}
The assertion follows. 
\end{proof}

\section{The Embedded  Homology  of Hypergraphs}\label{sec9.2}

In this section, we  study    the  embedded   homology   of hypergraphs. In Subsection~\ref{sec9.3.1}, we study the associated simplicial complex of hypergraphs. In Subsection~\ref{sec9.3.2}, we define the embedded  homology  of hypergraphs and prove some basic properties of the embedded homology. In Subsection~\ref{sec9.3.3}, we give a Mayer-Vietoris sequence for the embedded homology of hypergraphs.  


\subsection{The Associated Simplicial Complex of Hypergraphs}\label{sec9.3.1}

Given a hypergraph $\mathcal{H}$,  A.D. Parks and S.L. Lipscomb  \cite{parks} defined its {\bf associated simplicial complex} $K_\mathcal{H}$   to be the smallest simplicial complex  such that the hyperedges of  $\mathcal{H}$ is a subset of the simplices of $K_\mathcal{H}$.   Precisely,  the set of all simplices of $K_\mathcal{H}$  consists of  all the non-empty subsets $\tau\subseteq \sigma$,  for all $\sigma\in \mathcal{H}$ (cf. \cite[Lemma~8]{parks}).   All the hyperedges in $K_\mathcal{H}\setminus \mathcal{H}$ forms a hypergraph, which will be called the {\bf complement hypergraph} of $\mathcal{H}$ and denoted as $\mathcal{H}^c$,  in this paper.  

\smallskip

Firstly, we give a functor from the category of hypergraphs to the category of simplicial complexes, sending $\mathcal{H}$ to $K_\mathcal{H}$. A simplicial map from a simplicial complex  $K$ to a simplicial complex  $K'$ is a map $f$ sending a vertex of  $K$ to  a    vertex of $K'$ such that for any simplex $\sigma=\{v_0,\cdots,v_n\}$ of $K$, $f(\sigma)=\{f(v_0),\cdots,f(v_n)\}$ is a simplex of $K'$.  A morphism of hypergraphs from a hypergraph  $\mathcal{H} $ to a hypergraph $\mathcal{H}' $ is a map 
 $f$ sending a vertex of $\mathcal{H}$ to a vertex of $\mathcal{H}'$   such that whenever $\sigma=\{v_0,\cdots,v_n\}$ is a hyperedge of  $\mathcal{H}$, $f(\sigma)=\{f(v_0),\cdots,f(v_n)\}$ is a hyperedge of  ${\mathcal{H}'}$.    Given a morphism of hypergraphs $f: \mathcal{H} \longrightarrow \mathcal{H}' $,  we have a simplicial map $\tilde f: K_\mathcal{H}\longrightarrow K_{\mathcal{H}' }$  sending a simplex $\{v_0, v_1, \cdots, v_n\}$ in $K_\mathcal{H}$ to the simplex $\{f(v_0), f(v_1), \cdots,  f(v_n)\}$ in $K_{\mathcal{H}'} $.  Consequently, 
we have a functor $\mathcal{F}$ from the category of hypergraphs to the category of simplicial complexes,  sending a hypergraph $\mathcal{H} $ to its  associated  simplicial complex $K_\mathcal{H}$ and sending a
morphism $f$ of hypergraphs to its induced simplicial map $\tilde f$ of the corresponding  associated  simplicial complexes.  The functor $\mathcal{F}$ is the adjoint functor of the forgetful functor from the category of simplicial complexes to the category of hypergraphs.

\smallskip

Secondly, we prove some basic properties of the associated simplicial complexes, in the remaining part of this section.  
The next proposition gives some basic properties of the associated simplicial complexes of hypergraphs. 

\begin{proposition}\label{pr9.3.1.1}
Let $\mathcal{H}$ and $\mathcal{H}'$ be hypergraphs. Then
\begin{eqnarray}
K_{\mathcal{H}\cup\mathcal{H}'}=K_{\mathcal{H}}\cup K_{\mathcal{H}'},
\label{eq9.3.1.28}\\
\label{eq9.3.1.18}
K_{\mathcal{H}\cap\mathcal{H}'}\subseteq K_{\mathcal{H}}\cap K_{\mathcal{H}'}.
\end{eqnarray}
Moreover, the equality in (\ref{eq9.3.1.18}) holds if for any $\sigma\in\mathcal{H}$ and any $\sigma'\in\mathcal{H}'$, either $\sigma\cap\sigma'$ is empty or  $\sigma\cap\sigma'\in\mathcal{H}\cap\mathcal{H}'$. 
\end{proposition}
\begin{proof}

To verify (\ref{eq9.3.1.28}),  we notice that $\{v_0,\cdots,v_n\}\in K(\mathcal{H}\cup \mathcal{H}')$ if and only if $v_0,\cdots,v_n \in V_\mathcal{H}  \cup V_{\mathcal{H}'} $ and there exists a hyperedge $\sigma \in \mathcal{H} \cup \mathcal{H}'$ such that $\{v_0,\cdots ,v_n\} \subseteq \sigma$. That is, either $\sigma$ is a hyperedge of $\mathcal{H}$ and $\{v_0,\cdots ,v_n\} \subseteq \sigma$, or $\sigma$ is a hyperedge of $\mathcal{H}'$ and $\{v_0,\cdots ,v_n\} \subseteq \sigma$. Hence $\{v_0,\cdots,v_n\}\in K(\mathcal{H}\cup \mathcal{H}')$ if and only if either $\{v_0,\cdots,v_n\}\in K_{\mathcal{H}}$ or $\{v_0,\cdots,v_n\}\in K_{\mathcal{H}'}$.

To verify  (\ref{eq9.3.1.18}),  we notice that $\{v_0,\cdots,v_n\}\in K(\mathcal{H}\cap \mathcal{H}')$ if and only if $v_0,\cdots,v_n \in V_\mathcal{H}  \cap V_{\mathcal{H}' }$ and there exists a hyperedge $\sigma \in \mathcal{H} \cap \mathcal{H}'$ such that $\{v_0,\cdots ,v_n\} \subseteq \sigma$.   On the other hand, $\{ v_0,\cdots ,v_n\} \in K_\mathcal{H}\cap K_{\mathcal{H}'}$ if and only if there exist a hyperedge $\sigma\in  \mathcal{H}$ and a hyperedge $\sigma'\in  \mathcal{H}'$ such that $\{v_0,\cdots ,v_n\} \subseteq  \sigma$ and $\{v_0,\cdots ,v_n\} \subseteq  \sigma'$.  Hence $\{v_0,\cdots,v_n\}\in K(\mathcal{H}\cap \mathcal{H}')$ implies  $\{ v_0,\cdots ,v_n\} \in K_\mathcal{H}\cap K_{\mathcal{H}'}$.  In particular, the equality of  (\ref{eq9.3.1.18}) holds under our additional the assumption. 
\end{proof}

The next proposition gives a universal property of the associated simplicial complexes of   hypergraphs. 

\begin{proposition}\label{p9.3.1.1}
Let $\mathcal{H}$ be a hypergraph and let $i_\mathcal{H}: \mathcal{H}\longrightarrow K_\mathcal{H} $ be the canonical embedding of $\mathcal{H}$ into the associated simplicial complex. If there is a  simplicial complex $K$  and an  injective map $i: \mathcal{H}\longrightarrow K$, then there exists an injective  simplicial map $\phi$ such that $\phi\circ i_\mathcal{H}=i$, making the following diagram commute: 
\begin{eqnarray*}
\xymatrix{
\mathcal{H}\ar[r] ^{i}\ar[d]_{i_\mathcal{H}} & K\\
K_\mathcal{H}.\ar@{..>}[ru]_\phi
}
\end{eqnarray*}
\end{proposition}
\begin{proof}
  Let $\{v_0,\cdots,v_n\}$ be a hyperedge of $\mathcal{H}$. Then, since $i$ is injective, $\{i(v_0),\cdots,i(v_n)\}$ is an $n$-simplex of $K$. For any subset $\{v_{j_0},\cdots,v_{j_m}\}$ of $\{v_0,\cdots,v_n\}$,  since $K$ is a simplicial complex, we have that $\{i(v_{j_0}),\cdots,i(v_{j_m})\}$ is a simplex of $K$. On the other hand, by the definition,  any simplex of $K_\mathcal{H}$ is  of the form $\{v_{j_0},\cdots,v_{j_m}\}$ (here we regard $i_\mathcal{H}$ as the canonical inclusion and do not distinguish a hyperedge of $\mathcal{H}$ with its image in $K_\mathcal{H}$). Therefore, we obtain a simplicial map $\phi: K_\mathcal{H}\longrightarrow K$ sending the  simplex $\{v_{j_0},\cdots,v_{j_m}\}$ to the simplex  $\{i(v_{j_0}),\cdots,i(v_{j_m})\}$.  It follows that $\phi$ is injective and $\phi\circ i_\mathcal{H}=i$.
\end{proof}

The next proposition  is a consequence of  Proposition~\ref{p9.3.1.1}. 

\begin{Proposition}\label{c9.3.1.2}
 Let $\mathcal{H}$ be a hypergraph and let $K$ be a simplicial complex such that there is an  injective map $i: \mathcal{H}\longrightarrow K$.   Then  
\begin{eqnarray}\label{eq9.1.91}
\text{Inf}_*(G(\mathcal{H}_*),G(K_*))=\text{Inf}_*(G(\mathcal{H}_*),G((K_{\mathcal{H}})_*))
\end{eqnarray} 
and 
\begin{eqnarray}\label{eq9.1.92}
\text{Sup}_*(G(\mathcal{H}_*),G(K_*))=\text{Sup}_*(G(\mathcal{H}_*),G((K_{\mathcal{H}})_*)).
\end{eqnarray} 
\end{Proposition}

For simplicity, we denote the simplicial complex (\ref{eq9.1.91}) as $\text{Inf}_*(\mathcal{H})$ and denote the simplicial complex (\ref{eq9.1.92}) as $\text{Sup}_*(\mathcal{H})$. By Proposition~\ref{c9.3.1.2}, both $\text{Inf}_*(\mathcal{H})$ and $\text{Sup}_*(\mathcal{H})$ do not depend on the choice of the ambient simplicial complex $K$  that $\mathcal{H}$ is  embedded in.

\subsection{The Embedded  Homology  of Hypergraphs}\label{sec9.3.2}

Given a hypergraph $\mathcal{H}$ and an abelian group $G$, we define the $n$-th {\bf embedded homology} of $\mathcal{H}$ (with coefficients in $G$) as 
\begin{eqnarray*}
H_n(\mathcal{H})=H_n(\text{Inf}_*(\mathcal{H})).
\end{eqnarray*}
In particular, if $\mathcal{H}$ is a simplicial complex, then the embedded homology $H_*(\mathcal{H})$ coincides with the usual simplicial homology of $\mathcal{H}$.  In this subsection, we give some basic properties of the embedded homology of hypergraphs. Then we give an example illustrating the computation. 

\smallskip

The next proposition follows from Proposition~\ref{prop9.2.8} and Proposition~\ref{c9.3.1.2}. 

\begin{proposition}\label{p9.3.2.1}
Let $G$ be an abelian group. Let $\mathcal{H}$ be a hypergraph,   $K$ be a simplicial complex that $\mathcal{H}$ can be embedded in and   $\partial_*$ be the boundary map of $K$. Then  
\begin{eqnarray*}
H_n(\mathcal{H})
&\cong&H_n(\text{Sup}_*(\mathcal{H}))\nonumber\\
&\cong&\text{Ker}(\partial_n|_{G(\mathcal{H}_n)})/((G(\mathcal{H}_n))\cap \partial_{n+1}(G(\mathcal{H}_{n+1}))). 
\end{eqnarray*}
\end{proposition}




The next proposition gives a geometric interpretation of the $0$-th embedded homology. 

\begin{proposition}\label{pr9.3.5}
Let $\mathcal{H}$ be a hypergraph such that all the vertices  are hyperedges. Then $\mathcal{H}_0\cup\mathcal{H}_1$ is a simplicial complex. Moreover, with integral coefficients, $H_0(\mathcal{H})=\mathbb{Z}^{\oplus k}$ where $k$ is the number of connected components of $\mathcal{H}_0\cup\mathcal{H}_1$. 
\end{proposition}

\begin{proof}
Since all  the vertices  are hyperedges, we see that $\mathcal{H}_0\cup\mathcal{H}_1$ is a simplicial complex. Moreover,
\begin{eqnarray*}
H_0(\mathcal{H})&=&\text{Ker}(\partial_0|_{\mathbb{Z}(\mathcal{H}_0)})/\partial_1(\mathbb{Z}(\mathcal{H}_1))\cap \mathbb{Z}(\mathcal{H}_0)\\
&=&\mathbb{Z}(\mathcal{H}_0)/\partial_1(\mathbb{Z}(\mathcal{H}_1))\\
&=&H_0(\mathcal{H}_0\cup \mathcal{H}_1). 
\end{eqnarray*}
Since $H_0(\mathcal{H}_0\cup \mathcal{H}_1)=\mathbb{Z}^k$ where $k$ is the number of connected components of $\mathcal{H}_0\cup\mathcal{H}_1$, the assertion follows. 
\end{proof}

The next proposition gives   the top embedded homology of a hypergraph.

\begin{proposition}\label{pr9.3.top}
Let $\mathcal{H}$ be a hypergraph of dimension $n$, $n\geq 0$. Then $H_n(\mathcal{H})=H_n(K_\mathcal{H})$. 
\end{proposition} 

\begin{proof}
  Since the dimension of $\mathcal{H}$ is $n$,  $\mathcal{H}_i$ is empty for any $i\geq n+1$. Hence $K_\mathcal{H}$ is a simplicial complex of dimension $n$ and $\mathcal{H}_n=(K_\mathcal{H})_n$. Therefore, if we let $\partial_*$ be the boundary map of $K_\mathcal{H}$ and  take integral coefficients  for convenience,  then both $H_n(\mathcal{H})$ and $H_n(K_\mathcal{H})$ are $\text{Ker}\partial_n$. 
The assertion follows. 
\end{proof}

The next proposition gives the naturality property of the embedded  homology.

\begin{proposition}\label{pers1-prop2}
Let $f: \mathcal{H} \longrightarrow \mathcal{H}' $ be a morphism of hypergraphs.  Then   $f$ induces a homomorphism  of graded groups
$f_*: H_*( \mathcal{H})\longrightarrow H_*( \mathcal{H}')$.
\end{proposition}

\begin{proof}
The functor $\mathcal{F}$ sends $f$ to a simplicial map $\tilde f:K_\mathcal{H}\longrightarrow K_{\mathcal{H}'}$.  Let $\{v_0,$ $\cdots,$ $v_n\}$ be a simplex of $K_\mathcal{H}$. Given an abelian group $G$,  we let $\tilde f_G$ be the map sending $\{v_0,$ $\cdots,$ $v_n\}$ to  $\{f(v_0),$ $f(v_1),$ $\cdots,$  $f(v_n)\}$   if $f(v_0),$ $f(v_1),$ $\cdots,$  $f(v_n)$ are distinct, and sending  $\{v_0,$ $\cdots,$ $v_n\}$ to $0$ (the identity element of $G$) otherwise. Then by extending $\tilde f_G$ linearly over $G$,  we obtain  a chain map (still denoted as $\tilde f_G$) 
from the chain complex  $G({K_\mathcal{H}}_*)$ to the chain complex $G({K_{\mathcal{H}'}}_*)$.  Moreover, restricting the  chain map $\tilde f_G$ to the infimum chain complex, we obtain a chain map from $\text{Inf}_*(\mathcal{H})$ to $\text{Inf}_*(\mathcal{H}')$. This induces a homomorphism of graded groups $f_*: H_*( \mathcal{H})\longrightarrow H_*( \mathcal{H}')$.
 \end{proof}

The next example illustrates how to compute the embedded homology of hypergraphs. 

\begin{example}Let the hypergraph $\mathcal{H}=\{\{v_0\},\{v_1\},\{v_2\},\{v_0,v_1\},\{v_0,v_1,v_2\}\}$. Then its associated simplicial complex is 
\begin{eqnarray*}
K_\mathcal{H}=\{\{v_0\},\{v_1\},\{v_2\},\{v_0,v_1\},\{v_1,v_2\},\{v_0,v_2\},\{v_0,v_1,v_2\}\}.
\end{eqnarray*}
 The following picture illustrates $\mathcal{H}$ and $K_\mathcal{H}$:
 
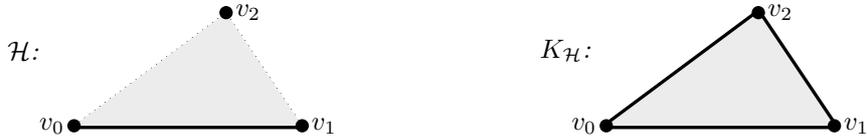
\bigskip
 
 \begin{center}
\begin{tikzpicture}
\coordinate [label=left:$v_0$]  (A) at (2,0); 
\coordinate [label=right:$v_1$]  (B) at (5,0); 
   \coordinate [label=right:$\mathcal{H}$:]  (F) at (1,1); 
 \draw [line width=2.5pt] (A) -- (B);
 \coordinate [label=right:$v_2$]  (C) at (4,1.5); 
 \draw [dotted] (B) -- (C);
  \draw [dotted] (A) -- (C);        
 \fill [gray!15!white] (A) -- (B) -- (C) -- cycle;
  \coordinate [label=left:$v_0$]  (A) at (9,0); 
  \coordinate [label=right:$v_1$]  (B) at (12,0); 
  \coordinate [label=right:$K_\mathcal{H}$:]  (F) at (8,1); 
 \draw [line width=2.5pt] (A) -- (B);
 \coordinate [label=right:$v_2$]  (C) at (11,1.5); 
 \draw[line width=2.5pt] (C) -- (B);
\draw [line width=2.5pt]   (A) -- (C);
\fill [gray!15!white] (A) -- (B) -- (C) -- cycle;
\fill (2,0) circle (2.5pt) (5,0) circle (2.5pt)  (4,1.5) circle (2.5 pt);        
\fill (9,0) circle (2.5pt) (12,0) circle (2.5pt)  (11,1.5) circle (2.5 pt);        
\end{tikzpicture}
\end{center}

\noindent Moreover, if we let $\partial_n$ be the boundary map of ${K}_\mathcal{H}$ and take integral coefficients, then  
\begin{eqnarray*}
H_0(\mathcal{H})
&=&\text{Ker}(\partial_0|_{\mathbb{Z}(\{v_0\},\{v_1\},\{v_2\})})/ \partial_1(\mathbb{Z}(\{v_0,v_1\}))\cap\mathbb{Z}(\{v_0\},\{v_1\},\{v_2\})\\
&=&\mathbb{Z}(\{v_0\},\{v_1\},\{v_2\})/\mathbb{Z}(\{v_1\}-\{v_0\})\cap \mathbb{Z}(\{v_0\},\{v_1\},\{v_2\})\\
&=&\mathbb{Z}\oplus \mathbb{Z },\\
 H_1(\mathcal{H})&=&\text{Ker}(\partial_1|_{\mathbb{Z}(\{v_0,v_1\})})/ \partial_2(\mathbb{Z}(\{v_0,v_1,v_2\}))\cap\mathbb{Z}(\{v_0,v_1\})\\
 &=&0,\\
H_2(\mathcal{H})&=&\text{Ker}(\partial_2|_{\mathbb{Z}(\{v_0,v_1,v_2\})})/ 0\cap\mathbb{Z}(\{v_0,v_1,v_2\})\\
 &=&0.
 \end{eqnarray*}
\end{example}

\subsection{The Mayer-Vietoris Sequence for the Embedded Homology of  Hypergraphs}\label{sec9.3.3}



In the  homology  theory of simplicial complexes, the Mayer-Vietoris sequence is an algebraic tool to  compute the homology  groups of simplicial complexes.  With the Mayer-Vietoris sequence, we are able  to reduce the homology of a complicated simplicial complex  $K$  into the homology of simpler simplicial complexes  $K_1,K_2\subseteq K$  such that  $K$  is the union of $K_1$ and $K_2$.  In this subsection, we give a Mayer-Vietoris sequence for the embedded homology of hypergraphs in Theorem~\ref{c9.3.3.1}. The Mayer-Vietoris sequence for the embedded homology of hypergraphs allows us to reduce the embedded homology of a complicated hypergraph to the embedded homology of simpler hypergraphs. 


\smallskip

Let $\mathcal{H} $ and $\mathcal{H}' $ be two hypergraphs. With the help of Proposition~\ref{pr9.2.18}, we have  the following short exact sequences  
\begin{eqnarray}
&0\longrightarrow \text{Inf}_*(\mathcal{H})\cap \text{Inf}_*(\mathcal{H}')\longrightarrow \text{Inf}_*(\mathcal{H})\oplus \text{Inf}_*(\mathcal{H}') \nonumber \\
&\longrightarrow \text{Inf}_*(\mathcal{H})+ \text{Inf}_*(\mathcal{H}')\longrightarrow 0, \label{9.4.8} \\
&0\longrightarrow  \text{Inf}_*(\mathcal{H})+ \text{Inf}_*(\mathcal{H}')\longrightarrow  \text{Inf}_*(\mathcal{H}\cup \mathcal{H}')\nonumber\\
&\longrightarrow \text{Inf}_*(\mathcal{H}\cup \mathcal{H}')/ (\text{Inf}_*(\mathcal{H})+ \text{Inf}_*(\mathcal{H}'))\longrightarrow 0.\label{9.4.88}
\end{eqnarray}
Moreover, we have the next proposition.

\begin{proposition}\label{p9.4.1}
Let $\mathcal{H}$ and $\mathcal{H}'$ be two hypergraphs such that     for any $\sigma\in\mathcal{H}$ and any $\sigma'\in \mathcal{H}'$, either $\sigma\cap\sigma'$ is empty or $\sigma\cap \sigma'\in \mathcal{H}\cap\mathcal{H}'$.  Then we have a short exact sequence
\begin{eqnarray*}
0\longrightarrow \text{Inf}_*(\mathcal{H})\cap \text{Inf}_*(\mathcal{H}')\longrightarrow \text{Inf}_*(\mathcal{H})\oplus \text{Inf}_*(\mathcal{H}') 
\longrightarrow \text{Inf}_*(\mathcal{H}\cup \mathcal{H}')\longrightarrow 0.
\end{eqnarray*}
\end{proposition}
\begin{proof}
We choose $K_{\mathcal{H}\cup\mathcal{H}'}$ as the ambient simplicial complex in Proposition~\ref{c9.3.1.2} to construct $\text{Inf}_*(\mathcal{H})$,  $\text{Inf}_*(\mathcal{H}')$, $\text{Inf}_*(\mathcal{H}\cap\mathcal{H}')$ and $\text{Inf}_*(\mathcal{H}\cup\mathcal{H}')$. 
Let $\alpha\in \text{Inf}_n(\mathcal{H}\cup \mathcal{H'})$.  Then 
\begin{eqnarray}\label{eq9.4.18}
\alpha\in G(\mathcal{H}_n\cup\mathcal{H}'_n) 
\end{eqnarray}
 and 
 \begin{eqnarray}\label{eq9.4.28}
 \partial_n \alpha\in G(\mathcal{H}_{n-1}\cup\mathcal{H}'_{n-1}).
 \end{eqnarray} 
Moreover, for any non-negative integer $i$,
 \begin{eqnarray}\label{eq9.4.3}
G(\mathcal{H}_{i}\cup\mathcal{H}'_{i})=G(\mathcal{H}_i)+G(\mathcal{H}'_i). 
 \end{eqnarray} 
 From (\ref{eq9.4.18}) and (\ref{eq9.4.3}), we can assume that $\alpha=\alpha_1+\alpha_2$ where $\alpha_1\in G(\mathcal{H}_n)$ and $\alpha_2\in G(\mathcal{H}'_n\setminus \mathcal{H}_n)$. 
 We can also assume that 
 \begin{eqnarray}\label{eq9.4.5}
 \partial_n \alpha_1=\beta_1+\gamma_1
 \end{eqnarray}
  where $\beta_1\in G(\mathcal{H}_{n-1})$, $\gamma_1\in G(\mathcal{H}^c_{n-1})$; and 
 \begin{eqnarray}\label{eq9.4.6}
  \partial_n \alpha_2=\beta_2+\gamma_2
  \end{eqnarray}
   where $\beta_2\in G(\mathcal{H}'_{n-1})$, $\gamma_2\in G(\mathcal{H}'^c_{n-1})$.  From (\ref{eq9.4.28}) - (\ref{eq9.4.6}),  we have
      \begin{eqnarray*}
 \gamma_1+\gamma_2\in G(\mathcal{H}_{n-1})+G(\mathcal{H}'_{n-1}).
   \end{eqnarray*} 
It follows that
       \begin{eqnarray}\label{eq9.4.10}
        \gamma_1+\gamma_2=0.
    \end{eqnarray}   
Suppose 
      \begin{eqnarray}\label{9.mv.1}
\gamma_1=\sum_t k_t d_i\sigma_t,\\
\label{9.mv.2}
 \gamma_2=\sum_s h_t d_i\sigma'_s,
      \end{eqnarray}
 where $k_t, h_s\in G$, $\sigma_t\in \mathcal{H}_n$,  $\sigma'_s\in \mathcal{H}'_n$, $d_i$'s are the face maps of $K_{\mathcal{H}\cup\mathcal{H}'}$,  $d_i\sigma_t\in \mathcal{H}^c_{n-1}$ and  $d_i\sigma'_s\in \mathcal{H}'^c_{n-1}$. 
We claim that for any summand $ k_t d_i\sigma_t$ of (\ref{9.mv.1}) and any summand $h_t d_i\sigma'_s$  of (\ref{9.mv.2}), 
\begin{eqnarray}\label{9.mv.3}
d_i\sigma_t\neq d_j\sigma_s'.
\end{eqnarray}
To prove (\ref{9.mv.3}), we  suppose to the contrary that $d_i\sigma_t= d_j\sigma_s'$ for some $i$, $j$, $t$ and $s$.  Then we have the next two cases. 

{\sc Case~1}. $\sigma_t=\sigma'_s$.

By multiplying certain coefficient in $G$, $\sigma_t$ is a summand of $\alpha_1$. Hence $\sigma_t\in \mathcal{H}$. And by multiplying another coefficient in $G$, $\sigma'_s$ is a summand of $\alpha_2$.  Hence $\sigma'_s\in\mathcal{H}'\setminus\mathcal{H}$. This contradicts $\sigma_t=\sigma_s'$.

{\sc Case~2}. $\sigma_t\neq \sigma'_s$.

Then $d_i\sigma_t=d_j\sigma'_s=\sigma_t\cap\sigma'_s$, which is non-empty. Moreover, since $\sigma_t\in\mathcal{H}$ and $\sigma_s'\in\mathcal{H}'$, we have $\sigma_t\cap\sigma'_s\in \mathcal{H}\cap \mathcal{H}'$.  Therefore, $d_i\sigma_t\in\mathcal{H}$ and $d_j\sigma_s'\in\mathcal{H}'$. This contradicts the assumption on $\gamma_1$ and $\gamma_2$.  

Summarising Case~1 and Case~2, we have (\ref{9.mv.3}).  It follows from (\ref{eq9.4.10}) - (\ref{9.mv.3}) that 
 $\gamma_1=\gamma_2=0$. 
Therefore,  the canonical homomorphism from $\text{Inf}_*(\mathcal{H})\oplus \text{Inf}_*(\mathcal{H}') 
$ to $ \text{Inf}_*(\mathcal{H}\cup \mathcal{H}')$ is an epimorphism.  Consequently, the short exact sequence (\ref{9.4.88}) is trivial,  and  the short exact sequence (\ref{9.4.8}) implies the assertion. 
\end{proof}

The following theorem is an immediate consequence of Proposition~\ref{p9.4.1}. 

\begin{theorem}\label{c9.3.3.1}
Let $\mathcal{H}$ and $\mathcal{H}'$ be two hypergraphs such that   for any $\sigma\in\mathcal{H}$ and any $\sigma'\in \mathcal{H}'$,  either $\sigma\cap\sigma'$ is empty or $\sigma\cap \sigma'\in \mathcal{H}\cap\mathcal{H}'$.  Then we have a long exact sequence of homology
\begin{eqnarray*}
\cdots\longrightarrow H_n(\mathcal{H}\cap\mathcal{H}')\longrightarrow H_n(\mathcal{H})\oplus H_n(\mathcal{H}') 
\longrightarrow H_n(\mathcal{H}\cup \mathcal{H}')\longrightarrow H_{n-1}(\mathcal{H}\cap\mathcal{H}')\longrightarrow\cdots.
\end{eqnarray*}
\end{theorem}

The long exact sequence given in Theorem~\ref{c9.3.3.1} is called the {\bf Mayer-Vietoris sequence} for the embedded homology of hypergraphs. 
 
\begin{remark}\label{r9.3.3.1}
Generally, let   $\mathcal{H}$ and $\mathcal{H}'$ be two arbitrary hypergraphs.  By (\ref{9.4.8}) and (\ref{9.4.88}), we have two long exact sequences of homology
\begin{eqnarray*}
&
\cdots\longrightarrow H_n(\mathcal{H}\cap\mathcal{H}')\longrightarrow H_n(\mathcal{H})\oplus H_n(\mathcal{H}') \\
&
\longrightarrow H_n(\text{Inf}_*(\mathcal{H})+ \text{Inf}_*(\mathcal{H}'))\longrightarrow H_{n-1}(\mathcal{H}\cap\mathcal{H}')\longrightarrow\cdots,\\
&
\cdots\longrightarrow H_n(\text{Inf}_*(\mathcal{H})+ \text{Inf}_*(\mathcal{H}'))\longrightarrow H_n(\mathcal{H}\cup \mathcal{H}') \\
&
\longrightarrow H_n(\mathcal{H}\cup \mathcal{H}'/(\text{Inf}_*(\mathcal{H})+ \text{Inf}_*(\mathcal{H}')))\longrightarrow H_{n-1}(\text{Inf}_*(\mathcal{H})+ \text{Inf}_*(\mathcal{H}'))\longrightarrow\cdots.
\end{eqnarray*}
\end{remark}

The next example illustrates how to use the Mayer-Vietoris sequence to simplify the computation of the embedded homology of hypergraphs. 

\begin{example}
(i). 
Let $\mathcal{H}$ and $\mathcal{H}'$ be hypergraphs such that (a).  for any $\sigma\in\mathcal{H}$ and any $\sigma'\in \mathcal{H}'$,  either $\sigma\cap\sigma'$ is empty or $\sigma\cap \sigma'\in \mathcal{H}\cap\mathcal{H}'$;  (b).  the intersection of $\mathcal{H}$ and $\mathcal{H}'$ is a disjoint union of standard simplicial complexes
\begin{eqnarray*}
\mathcal{H}\cap\mathcal{H}'=\coprod_{i=1}^k \Delta[n_i]. 
\end{eqnarray*}
Then  since each $\Delta[n_i]$ is contractible,  for any $n\geq 1$,
\begin{eqnarray*}
H_n(\mathcal{H}\cap\mathcal{H}')=0.  
\end{eqnarray*}
Hence by Theorem~\ref{c9.3.3.1}, for any $n\geq 2$, 
 \begin{eqnarray*}
 H_n(\mathcal{H}\cup\mathcal{H}')\cong H_n(\mathcal{H})\oplus H_n(\mathcal{H}'). 
 \end{eqnarray*}

(ii). Suppose $\mathcal{H}^j$, $1\leq j\leq m$, is a sequence of hypergraphs such that for any $j_1\neq j_2$,  
(a). 
 for any $\sigma\in\mathcal{H}^{j_1}$ and any $\sigma'\in \mathcal{H}^{j_2}$,  either $\sigma\cap\sigma'$ is empty or $\sigma\cap \sigma'\in \mathcal{H}^{j_1}\cap\mathcal{H}^{j_2}$;
 (b). 
 $\mathcal{H}^{j_1}\cap\mathcal{H}^{j_2}$ is a disjoint union of standard simplicial complexes.
  Then by (i) and an induction on $m$, for any $n\geq 2$, 
 \begin{eqnarray*}
 H_n(\cup_{j=1}^m\mathcal{H}^j)\cong \oplus_{j=1}^mH_n(\mathcal{H}^j). 
 \end{eqnarray*}
 
 (iii). A concrete example of (ii) is as follows. For each $j=0,1,2,3$, let 
 \begin{eqnarray*}
 \mathcal{H}^j=\Delta[v_0,v_1,v_2,v_3]\cup \{\{v_0,\ldots, \hat{v_j},\ldots,v_3,w_j\} \}  
 \end{eqnarray*}
 where $\hat{v_j}$ stands for removing $v_j$.  
 Then by (ii), with integral coefficients, 
 \begin{eqnarray*}
 H_2(\cup_{j=1}^3\mathcal{H}^j)\cong \oplus _{j=1}^3 H_2(\mathcal{H}^j)\cong 0. 
 \end{eqnarray*}
\end{example}



 
\section{The Persistent Embedded  Homology  of Hypergraphs and Mayer-Vietoris Sequence} \label{sec9.7}

Persistent  homology  has been used to study the  invariant topological structures of sequences of topological objects.  Some algorithms to compute  persistent homology of simplicial complexes have been given by A. Zomorodian and G. Carlsson \cite{disc2004}.  In this section, we study the persistent embedded  homology  of sequences of hypergraphs and give a persistent version of Mayer-Vietoris sequence for the persistent embedded homology of hypergraphs, in Theorem~\ref{c9.4.8}.

\smallskip
 
Before studying the persistent embedded homology of hypergraphs, we review some  definitions.   A {\bf persistence complex}  $\mathcal{C}=\{C^i_*,\varphi_i\}_{i\geq 0}$ is a family of chain complexes $\{C^i_*\}_{i\geq 0}$, together with chain maps $\varphi_i: C_*^i\longrightarrow C^{i+1}_*$ (cf. \cite[Definition~3.1]{disc2004}).  A {\bf persistence abelian group } $\mathcal{G}=\{G^i,\psi_i\}_{i\geq 0}$ is a family of abelian groups $G^i$, together with group homomorphisms $\psi_i: G^i\longrightarrow G^{i+1}$. 
A persistence complex $\mathcal{C}$ (resp. a persistence abelian group $\mathcal{G}$) is called of {\bf finite type} if each component complex $C^i$ (resp. $G^i$) is a finitely generated abelian group and there exists an integer $N$ such that the maps $\varphi_i$ (resp. $\psi_i$) are isomorphisms for all $i\geq N$ (cf. \cite[Definition~3.3]{disc2004}).

\smallskip

Now we turn to the persistent embedded homology of hypergraphs.  We consider a sequence of hypergraphs with morphisms
\begin{eqnarray}\label{eq9.4.add}
\mathcal{H}^0\overset{f_0}{\longrightarrow}\mathcal{H}^1\overset{f_1}{\longrightarrow}\mathcal{H}^2\overset{f_2}{\longrightarrow}\cdots\overset{f_{i-1}}{\longrightarrow}\mathcal{H}^i\overset{f_i}{\longrightarrow}\cdots. 
\end{eqnarray}
Given an abelian group $G$,  the sequence (\ref{eq9.4.add}) induces  persistence complexes
\begin{eqnarray}\label{eq9.4.81}
&&\{\text{Inf}_*(G(\mathcal{H}^i_*),G(({K_{\mathcal{H}^i}})_* )), \tilde f_{i,G}\}_{i\geq 0},\\
&&\{\text{Sup}_*(G(\mathcal{H}^i_*),G(({K_{\mathcal{H}^i}})_* )), \tilde f_{i,G}\}_{i\geq 0}. \label{eq9.4.82}
\end{eqnarray}
Here  $\tilde f_{i,G}$ are the maps given  in the proof of Proposition~\ref{pers1-prop2}. Moreover, 
both (\ref{eq9.4.81}) and (\ref{eq9.4.82}) induce  a persistence  abelian group
\begin{eqnarray}\label{pers-ho}
H_*(\mathcal{H}^0)\overset{(f_0)_*}{\longrightarrow}H_*(\mathcal{H}^1)\overset{(f_1)_*}{\longrightarrow}H_*(\mathcal{H}^2)\overset{(f_2)_*}{\longrightarrow}\cdots. 
\end{eqnarray}
  We call the sequence (\ref{pers-ho}) the {\bf persistent embedded homology}   of  the sequence (\ref{eq9.4.add}).   

 \smallskip
 
A  {\bf filtration} of hypergraphs is a sequence (\ref{eq9.4.add}) such that each $f_i$, $i\geq 0$, is injective.  
Given a hypergraph $\mathcal{H}$, there are canonical ways  to construct a filtration of hypergraphs.  We give one way to construct a filtration as follows.

Let $d$ be a distance function on the set of vertices $V_\mathcal{H}$ of a hypergraph $\mathcal{H}$.  For any $r>0$,  we let $K(\mathcal{H},r)$ be the simplicial complex consisting of all simplices of the form
\begin{eqnarray*}
\{v_0,\cdots, v_n\mid n\geq 0,  v_0,\cdots,v_n\in V_\mathcal{H} \text{ and } d(v_i,v_j)< r\} 
\end{eqnarray*}
and  let the hypergraph 
\begin{eqnarray*}
\mathcal{H}(r)=\mathcal{H}\cap K(\mathcal{H},r).
\end{eqnarray*}
 Then for any increasing sequence of numbers $0<r_1<r_2<\cdots<r_k<\cdots$, we  get a filtration of hypergraphs
\begin{eqnarray}\label{eq.fil.1}
\mathcal{H}(r_1)\subseteq \mathcal{H}(r_2)\subseteq\cdots\subseteq \mathcal{H}(r_k)\subseteq\cdots\subseteq \mathcal{H}. 
\end{eqnarray}
Since all hypergraphs are assumed to have finite vertices,  there exists a positive $R$ such that for any $r>R$, $\mathcal{H}(r)=\mathcal{H}(R)$.  Therefore, the persistent  homology  of the filtration (\ref{eq.fil.1}) is of finite type. 
Moreover, if  $\lim_{k\to\infty}r_k=\infty$,  then $\mathcal{H}(r_k)$  converges to $\mathcal{H}$. In this case,  the persistent   embedded homology $H_*(\mathcal{H}(r_k))$ converges to $H_*(\mathcal{H})$.

\smallskip

Generalising the Mayer-Vietoris sequence of the embedded homology to the persistent embedded homology,  the  next theorem follows from Proposition~\ref{c9.3.3.1}. 

\begin{theorem}\label{c9.4.8}
Let $\mathcal{H}$ and $\mathcal{H}'$ be two hypergraphs such that   for any $\sigma\in\mathcal{H}$ and any $\sigma'\in \mathcal{H}'$, either $\sigma\cap\sigma'$ is empty or $\sigma\cap \sigma'\in \mathcal{H}\cap\mathcal{H}'$.  Let $d$ be a distance function on $\mathcal{H}\cup \mathcal{H}'$. Then for any $0<r_1<r_2<\cdots<r_k<\cdots$,  we have  long exact sequences of embedded homology in each row of the following commutative diagram
{
\small 
\begin{eqnarray*}
\xymatrix{
\cdots H_n(\mathcal{H}\cap\mathcal{H}')\ar[r] & H_n(\mathcal{H})\oplus H_n(\mathcal{H}') 
\ar[r] & H_n(\mathcal{H}\cup \mathcal{H}')\ar[r] & H_{n-1}(\mathcal{H}\cap\mathcal{H}') \cdots\\
 \cdots \ar[u]& \cdots \ar[u]&\cdots \ar[u]&\cdots \ar[u] \\
\cdots H_n(\mathcal{H}(r_k)\cap\mathcal{H}'(r_k))\ar[r]\ar[u] & H_n(\mathcal{H}(r_k))\oplus H_n(\mathcal{H}'(r_k)) 
\ar[r] \ar[u]& H_n(\mathcal{H}(r_k)\cup \mathcal{H}'(r_k))\ar[r]\ar[u] & H_{n-1}(\mathcal{H}(r_k)\cap\mathcal{H}'(r_k))\ar[u]  \cdots\\
\cdots \ar[u]& \cdots \ar[u]&\cdots \ar[u]&\cdots \ar[u] \\
\cdots H_n(\mathcal{H}(r_1)\cap\mathcal{H}'(r_1))\ar[r]\ar[u] & H_n(\mathcal{H}(r_1))\oplus H_n(\mathcal{H}'(r_1)) 
\ar[r] \ar[u]& H_n(\mathcal{H}(r_1)\cup \mathcal{H}'(r_1))\ar[r]\ar[u] & H_{n-1}(\mathcal{H}(r_1)\cap\mathcal{H}'(r_1))\ar[u]  \cdots
}
\end{eqnarray*}
}
\end{theorem}
\begin{proof}
We first claim that for each $r>0$,  $\sigma\in \mathcal{H}(r)$ and  $\sigma'\in \mathcal{H}'(r)$ imply that either $\sigma\cap\sigma'$ is empty or $\sigma\cap\sigma'\in \mathcal{H}(r)\cap\mathcal{H}'(r)$. To prove the claim, we choose any $\sigma\in \mathcal{H}(r)$ and any $\sigma'\in \mathcal{H}'(r)$ such that $\sigma\cap\sigma'$ is non-empty.  Then $\sigma\cap\sigma'\in \mathcal{H}\cap\mathcal{H}'$. Moreover,  for any vertices $v,w\in \sigma\cap\sigma'$, we have $d(v,w)<r$.  Hence $\sigma\cap\sigma'\in K(\mathcal{H}\cup\mathcal{H}', r)$. Since 
\begin{eqnarray*}
\mathcal{H}(r)\cap\mathcal{H}'(r)&=&(K(\mathcal{H}, r)\cap \mathcal{H})\cap(K(\mathcal{H}', r)\cap\mathcal{H}')\\
&=&(K(\mathcal{H}\cup\mathcal{H}', r)\cap \mathcal{H})\cap(K(\mathcal{H}\cup\mathcal{H}', r)\cap\mathcal{H}')\\
&=&K(\mathcal{H}\cup\mathcal{H}', r)\cap (\mathcal{H}\cap\mathcal{H}'), 
\end{eqnarray*}
we see that $\sigma\cap\sigma'\in \mathcal{H}(r)\cap\mathcal{H}'(r)$. The claim is obtained.  By Proposition~\ref{c9.3.3.1}, each row of the above commutative diagram is a long exact sequence. With the help of the naturality property of the embedded homology given in Proposition~\ref{pers1-prop2}, the assertion follows. 
\end{proof}



\section
{Applications of the Associated Simplicial Complex  and the  Embedded Homology in Acyclic  Hypergraphs}\label{sec9.5}

Acyclic  Hypergraphs is an important family of hypergraphs.  A hypergraph $\mathcal{H}$ is said to be acyclic if $\mathcal{H}$ can be reduced to an empty set by repeatedly applying the following two operations:
\begin{quote}
(O1). if $v$ is a vertex that belongs to only one hyperedge, then delete $v$ from the hyperedge containing it; 

(O2). if $\sigma\subsetneq \sigma'$ are two hyperedges, then delete $\sigma$ from $\mathcal{H}$. 
\end{quote}
\noindent The notion of acyclic hypergraphs was firstly introduced as an analogue of trees in graphs (cf. \cite{berge}).  It is the mathematical model for  database schemas  in database theory (cf. \cite{acyclic1}).   In this section, we study acyclic hypergraphs by using the associated simplicial complexes and the embedded homology.


\subsection{The Associated Simplicial Complex  of Acyclic Hypergraphs}

In this subsection, we  strengthen \cite[Theorem~9]{parks}  and give a necessary condition and a sufficient condition on the associated simplicial complexes for the acyclic property of hypergraphs
, in Theorem~\ref{th9.5.2}. Then we  give some examples.

\smallskip

Characterising the associated simplicial complexes of acyclic hypergraphs, the following theorem is proved in \cite{parks}. 

\begin{theorem}\cite[Theorem~9]{parks}\label{th9.5.1}
Let $\mathcal{H}$ be a connected acyclic hypergraph. Then with integral coefficients, $H_n(K_\mathcal{H})$ is zero when $n\geq 1$ and is $\mathbb{Z}$ when $n=0$. 
\end{theorem}
\begin{remark}
For a general acyclic hypergraph $\mathcal{H}$, it follows from Theorem~\ref{th9.5.1} that $H_n(K_\mathcal{H})$ is zero when $n\geq 1$ and is $\mathbb{Z}^{\oplus k}$ when $n=0$ where $k$ is the number of connected components of $K_\mathcal{H}$.  
\end{remark}

To characterise acyclic hypergraphs more precisely,  we strengthen  Theorem~\ref{th9.5.1}  and prove the next theorem.

\begin{theorem}\label{th9.5.2}
Let $\mathcal{H}$ be a hypergraph. 

(a). If $\mathcal{H}$ is acyclic, then $K_\mathcal{H}$ is acyclic as well,  and $K_\mathcal{H}$  has the homotopy type of  finite discrete points.  

(b). If  $K_\mathcal{H}$ is the associated simplicial complex of a finite disjoint union of simplices whose sets of vertices are mutually non-intersecting, then $\mathcal{H}$ is acyclic. 
\end{theorem}

The following corollary is a consequence of Theorem~\ref{th9.5.2}~(a).

\begin{corollary}\label{co9.5.1.1}
Let $\mathcal{H}$ be an acyclic hypergraph. If $n\geq 1$ is the dimension of $\mathcal{H}$, then $H_n(\mathcal{H})=0$. 
\end{corollary}
\begin{proof}
By Proposition~\ref{pr9.3.top},  we have $H_n(\mathcal{H})=H_n(K_\mathcal{H})$. Since $n\geq 1$, by Theorem~\ref{th9.5.2}~(a), we have $H_n(K_\mathcal{H})=0$. The assertion follows. 
\end{proof}

The following corollary is a particular case of Theorem~\ref{th9.5.2}~(b).

\begin{corollary}
Let $\mathcal{H}$ be a hypergraph such that $K_\mathcal{H}$ is the associated simplicial complex of a simplex. Then $\mathcal{H}$ is acyclic. 
\end{corollary}


In order to prove Theorem~\ref{th9.5.2}, we consider the following operation
\begin{quote}
(O1)'. if $v$ is a vertex that belongs to only one hyperedge consisting of at least two vertices, then delete $v$ from the hyperedge containing it. 
\end{quote}
We  prove the next  lemmas. 

\begin{lemma}\label{l9.5.5}
 A hypergraph $\mathcal{H}$ is acyclic if and only if $\mathcal{H}$ can be reduced to   finite discrete points by finite steps of the operations (O1)' and (O2). 
\end{lemma}
\begin{proof}
The assertion follows by a simple observation. 
\end{proof}

\begin{lemma}\label{l9.5.4}
Given a hypergraph $\mathcal{H}$, let $\mathcal{H}'$ be a hypergraph obtained from $\mathcal{H}$ by arbitrary finite steps of the operations (O1)' and (O2). Then $K_\mathcal{H}$ and $K_{\mathcal{H}'}$ are homotopy equivalent. 
\end{lemma}
\begin{proof}
The assertion follows from a geometric observation that  the operation (O2) on  $\mathcal{H}$ does not  change $K_\mathcal{H}$.  And the operation (O1)' on  $\mathcal{H}$ gives a deformation retract of $K_{\mathcal{H}}$,  hence it does not  change the homotopy type of $K_\mathcal{H}$. 
\end{proof}

Now we turn to prove Theorem~\ref{th9.5.2}.

\begin{proof}[Proof of Theorem~\ref{th9.5.2}~(a)]
Suppose  $\mathcal{H}$ is acyclic. Then $\mathcal{H}$ can be reduced to the empty set after finite steps of the operations (O1) and (O2). Since $K_\mathcal{H}$ is obtained by adding all the  non-empty  subsets  $\tau\subsetneq\sigma$ for all $\sigma\in \mathcal{H}$ as hyperedges,   we see that after finite steps of (O2), $K_\mathcal{H}$ can be reduced to $\mathcal{H}$. Hence $K_\mathcal{H}$ can be reduced to  the empty set after finite steps of the operations (O1) and (O2).  This implies $K_\mathcal{H}$ is acyclic.  Moreover, it follows from Lemma~\ref{l9.5.5} and Lemma~\ref{l9.5.4} that $K_\mathcal{H}$ has the homotopy type of  finite discrete points. 
\end{proof}

\begin{proof}[Proof of Theorem~\ref{th9.5.2}~(b)]
We divide the proof into two steps.

{\sc Step~1. }
We assume that  $K_\mathcal{H}$ is the associated simplicial complex of an  $n$-simplex.  Then there exists exactly one $n$-simplex in $K_\mathcal{H}$, denoted as   $\Delta^n$.  Since $K_\mathcal{H}$ is the smallest simplicial complex containing $\mathcal{H}$, we see that $\Delta^n$ is the exactly one hyperedge in $\mathcal{H}_n$. Since $\Delta^n$ consists of all the vertices of $K_\mathcal{H}$, it consists of all the vertices of $\mathcal{H}$. Hence for any hyperedge $\tau\in \mathcal{H}$, we have $\tau\subseteq \Delta^n$.  By applying (O2) repeatedly, all the hyperedges $\tau\subsetneq \Delta^n$ of $\mathcal{H}$ can be deleted and $\mathcal{H}$ can be reduced to  the hypergraph consisting of only one  hyperedge $\Delta^n$.  Thus $\mathcal{H}$ is acyclic. 

{\sc Step~2. }
We assume that $K_\mathcal{H}$ is the associated simplicial complex of a disjoint union of  simplices $\coprod_{j=1}^m \Delta^{n_j}$, where  for any distinct $i$ and $j$, their  sets of vertices  $V_{\Delta^{n_i}}$ and  $V_{\Delta^{n_j}}$ are non-intersecting.  Then following the notation of Example~\ref{ex9.1.1} and with the help of  Proposition~\ref{pr9.3.1.1}, 
\begin{eqnarray*}
K_\mathcal{H}=\coprod_{j=1}^m  \Delta[{n_j}]. 
\end{eqnarray*}
For each $j=1,\cdots,m$, we let 
\begin{eqnarray*}
\mathcal{H}(j)=\mathcal{H}\cap  \Delta[{n_j}]. 
\end{eqnarray*}
Then $\Delta[n_j]$ is the associated simplicial complex of $\mathcal{H}(j)$.  By {\sc Step~1}, we see that $\mathcal{H}(j)$ is acyclic. Since  
$\mathcal{H}$ is the disjoint union of  $\mathcal{H}(j)$'s, we see that $\mathcal{H}$ 
 is acyclic as well.  
\end{proof}

The converse of Theorem~\ref{th9.5.2}~(a) is not true. We give such examples as follows. 
\begin{example}\label{pr9.5.8}
Let $n\geq 3$.  
Then following the notations in Example~\ref{ex9.1.1}, we have that for any $0\leq i\leq n$, the hypergraph
\begin{eqnarray}\label{eq9.5.9}
\mathcal{H}=\{\Delta^{n-1}_j\mid 0\leq j\leq n, j\neq i\}
\end{eqnarray}
 is not acyclic, while $K_\mathcal{H}$ has the homotopy type of a single point. 
 \end{example} 
 
The hypergraph $\mathcal{H}$ in Example~\ref{pr9.5.8} is the collection of $(n-1)$-faces of $\Delta^n$ excluding the $i$-th $(n-1)$-face. The following picture shows the case $n=3$, $i=3$ in Example~\ref{pr9.5.8}. 
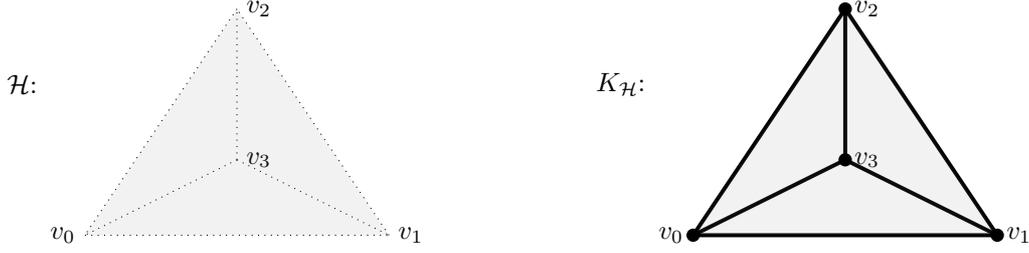
 \begin{center}
\begin{tikzpicture}
\coordinate [label=left:$v_0$]    (A) at (2,0); 
 \coordinate [label=right:$v_1$]   (B) at (6,0); 
 \coordinate  [label=right:$v_2$]   (C) at (4,3); 
\coordinate  [label=right:$v_3$]   (D) at (4,1); 
 \coordinate[label=left:$\mathcal{H}$:] (F) at (1.5,2);
 \draw [dotted] (A) -- (B);
 \draw [dotted] (B) -- (C);
 \draw [dotted] (C) -- (A);
  \draw [dotted] (D) -- (A);
\draw [dotted] (D) -- (B);
\draw [dotted] (D) -- (C);
\fill [fill opacity=0.1][gray!100!white] (A) -- (B) -- (D) -- cycle;
\fill [fill opacity=0.1][gray!100!white] (C) -- (B) -- (D) -- cycle;
\fill [fill opacity=0.1][gray!100!white] (A) -- (C) -- (D) -- cycle;
\coordinate [label=left:$v_0$]    (X) at (10,0); 
\coordinate [label=right:$v_1$]   (Y) at (14,0); 
\coordinate  [label=right:$v_2$]   (Z) at (12,3); 
\coordinate  [label=right:$v_3$]   (U) at (12,1); 
\coordinate[label=left:$K_\mathcal{H}$:] (F) at (9.5,2);
\draw [line width=1.5pt] (X) -- (Y);
\draw[line width=1.5pt] (Y) -- (Z);
\draw [line width=1.5pt] (Z) -- (X);
\draw [line width=1.5pt] (U) -- (X);
 \draw [line width=1.5pt] (U) -- (Y);
\draw [line width=1.5pt] (U) -- (Z);
\fill [fill opacity=0.1] [gray!100!white] (Y) -- (Z) -- (U) -- cycle;
\fill  [fill opacity=0.1][gray!100!white] (Z) -- (X) -- (U) -- cycle;
\fill  [fill opacity=0.1][gray!100!white] (Y) -- (X) -- (U) -- cycle;
 \fill (10,0) circle (2.5pt) (14,0) circle (2.5pt)  (12,3) circle (2.5pt)     (12,1)  circle   (2.5pt) ;    
 \end{tikzpicture}
\end{center}
\begin{proof}[Proof of Example~\ref{pr9.5.8}] 
Given a fixed $i$, by (\ref{eq9.5.9}), we see that $v_i$ belongs to $n$ hyperedges of $\mathcal{H}$ and for any $j\neq i$,  $v_j$ belongs to $n-1$ hyperedges of $\mathcal{H}$.  Since $n\geq 3$, each vertex of $\mathcal{H}$ belongs to at least two hyperedges.  Hence $\mathcal{H}$ cannot be reduced by (O1).  On the other hand, it is clear that for any distinct $j$ and  $l$,  $\Delta^{n-1}_j$ is not contained in $\Delta^{n-1}_l$. Hence $\mathcal{H}$ cannot be reduced by (O2).  Therefore, $\mathcal{H}$ cannot be reduced  by either (O1) or (O2).  Hence $\mathcal{H}$ is not acyclic. 
\end{proof}

The converse of Theorem~\ref{th9.5.2}~(b) is also not true.  The following is such an example.  
 
\begin{example}
Let $\mathcal{H}=\{\{v_0,v_1,v_2\},\{v_1,v_2,v_3\}\}$. Then $\mathcal{H}$ is acyclic, while 
\begin{eqnarray*}
K_\mathcal{H}=\{\{v_0,v_1,v_2\},\{v_1,v_2,v_3\}, \{v_0,v_1\},\{v_1,v_2\},\{v_0,v_2\},\{v_1,v_3\},\{v_2,v_3\},\{v_0\},\{v_1\},\{v_2\},\{v_3\}\}
\end{eqnarray*}
is connected and is not the associated simplicial complex of any single simplex. 
\end{example} 

\subsection{Hypergraphs Whose Associated Simplicial Complex Is $\Delta[n]$}\label{subsec9.5.2}

A particular family of acyclic hypergraphs is the hypergraphs whose associated simplicial complexes are $\Delta[n]$. In this section, we study the embedded homology of this family of acyclic hypergraphs.

\smallskip

Let $\mathcal{H}$ be a hypergraph whose associated simplicial complex is $\Delta[n]$, $n\geq 2$.  Then $\mathcal{H}$ is acyclic. Besides the triviality of $H_n(\mathcal{H})$  given  in Corollary~\ref{co9.5.1.1},  the following proposition shows the triviality of $H_{n-1}(\mathcal{H})$. 

\begin{proposition}\label{ex9.5.2.1}
Let $\mathcal{H}$ be a   hypergraph of dimension $n$ such that the associated simplicial complex of $\mathcal{H}$ is $\Delta[n]$, $n\geq 2$. Then   $H_{n-1}(\mathcal{H})$ is zero. 
\end{proposition}

\begin{proof}
For simplicity, we take integral coefficients. Let $\partial_k$, $k=0,1,\cdots,n$, be the boundary maps of $\Delta[n]$. Then  for any $k=0,1,\cdots,n$, 
\begin{eqnarray}\label{eq9.5.2.1}
H_k(\mathcal{H})=\text{Ker}(\partial_k|_{\mathbb{Z}(\mathcal{H}_k)})/(\mathbb{Z}(\mathcal{H}_k)\cap \partial_{k+1} \mathbb{Z}(\mathcal{H}_{k+1})). 
\end{eqnarray}
Since the associated simplicial complex of $\mathcal{H}$ is $\Delta[n]$, we have $\mathcal{H}_n=\{\Delta^n\}$ and 
\begin{eqnarray}
\partial_n\Delta^n&=&\sum_{j=0}^n(-1)^j \Delta^{n-1}_j.\label{eq9.5.2.2}
\end{eqnarray}
Since $n\geq 2$,  we have
\begin{eqnarray}\label{eq9.5.2.8}
\text{Ker}\partial_{n-1}=\mathbb{Z}(\sum_{j=0}^n\Delta^{n-1}_j). 
\end{eqnarray}
Moreover, since
$
\mathcal{H}_{n-1}\subseteq \{\Delta_j^{n-1}\mid j=0,1,\cdots,n\} 
$, we consider the next two cases. 

{\sc Case~1}. $\mathcal{H}_{n-1}\subsetneq \{\Delta_j^{n-1}\mid j=0,1,\cdots,n\}. 
$

Then by (\ref{eq9.5.2.8}),  we have 
\begin{eqnarray*}
\text{Ker}(\partial_{n-1}|_{\mathbb{Z}(\mathcal{H}_{n-1})})=0. 
\end{eqnarray*}
It follows from (\ref{eq9.5.2.1}) that $H_{n-1}(\mathcal{H})=0$. 

{\sc Case~2}. $\mathcal{H}_{n-1}= \{\Delta_j^{n-1}\mid j=0,1,\cdots,n\}. 
$

Then by (\ref{eq9.5.2.8}),  we have 
\begin{eqnarray}\label{eq9.5.2.3}
\text{Ker}(\partial_{n-1}|_{\mathbb{Z}(\mathcal{H}_{n-1})})=\mathbb{Z}(\sum_{j=0}^n\Delta^{n-1}_j). 
\end{eqnarray}
Moreover, with the help of (\ref{eq9.5.2.2}), we have 
\begin{eqnarray}\label{eq9.5.2.4}
\partial_n(\mathbb{Z}(\mathcal{H}_n))=\mathbb{Z}(\sum_{j=0}^n\Delta^{n-1}_j),  
\end{eqnarray}
which is a submodule of $\mathbb{Z}(\mathcal{H}_{n-1})$. 
It follows from (\ref{eq9.5.2.1}),  (\ref{eq9.5.2.3}) and   (\ref{eq9.5.2.4}) that $H_{n-1}(\mathcal{H})=0$. 

Summarising {\sc Case~1} and {\sc Case~2}, we have $H_{n-1}(\mathcal{H})=0$.   The assertion follows. 
\end{proof}

For any hypergraph $\mathcal{H}$, we will construct an acyclic hypergraph $\mathcal{H}'$ which contains $\mathcal{H}$ and has the same embedded homology with $\mathcal{H}$, in the next theorem. 

\begin{theorem}\label{pr9.5.2.3}
For any hypergraph $\mathcal{H}$, there exists an acyclic hypergraph $\mathcal{H}'$ such that

(i). $K_{\mathcal{H}'}=\Delta[n]$ for some $n\geq 2$;

(ii).   $\mathcal{H}\subsetneq \mathcal{H}'$;

(iii). $H_*(\mathcal{H}')\cong H_*(\mathcal{H})$.
\end{theorem}
\begin{proof}
Given a hypergraph $\mathcal{H}$, by adding  extra vertices $x,y$ to $V_\mathcal{H}$ and letting $\sigma=V_\mathcal{H}\cup\{x,y\}$, we have a hypergraph $\mathcal{H}'=\mathcal{H}\cup\{\sigma\}$, which satisfies (ii). Any hyperedge of $\mathcal{H}'$ is a subset of $\sigma$.  Hence if $\sigma$ is an $n$-hyperedge, then $K_{\mathcal{H}'}=\Delta[n]$. We obtain (i).  Moreover, $\mathcal{H}'_{i}=\mathcal{H}_i$ for any $0\leq i\leq n-2$,   $\mathcal{H}'_{n-1}$ is empty and $\mathcal{H}_n'=\{\sigma\}$.  Therefore, $H_i(\mathcal{H}')\cong H_i(\mathcal{H})$ for any $0\leq i\leq n-2$, while both $H_{n-1}(\mathcal{H}')$ and $H_n(\mathcal{H}')$ are zero. We obtain (iii). 
\end{proof}

The embedded homology of hypergraphs gives richer information than the associated simplicial complex.  An acyclic hypergraph whose associated simplicial complex is $\Delta[n]$ may have  highly non-trivial embedded homology,  as shown in  the next theorem. 

\begin{theorem} \label{th9.5.2.8}
For any $m\geq 1$ and any finitely-generated abelian groups $G_1,\cdots, G_{m}$,   there exists an acyclic hypergraph $\mathcal{H}$ such that 

(i). $K_\mathcal{H}=\Delta[n]$, where $n$ is the dimension of $\mathcal{H}$;

(ii). $n$ is less than or equal to $m+3$;

(iii).  $H_i(\mathcal{H})=G_i$ for $1\leq i\leq m$;

(iv). $\mathcal{H}_0\cup\mathcal{H}_1$ is connected.
\end{theorem}  

\begin{proof}
For any finitely-generated abelian groups $G_1,\cdots, G_{m}$,  we let $M(G_i,i)$ be the Moore space (cf.  \cite[p. 143]{hatcher}) such that 
\begin{itemize}
\item
 $H_i(M(G_i, i))=G_i$; 
\item
 $H_j(M(G_i,i))=0$ for any $j\neq i$, $j\geq 1$;
\item
 $M(G_i,i)$ is connected; 
\item
 $M(G_i,i)$ has cells only in dimension $0$, $i$ and $i+1$.
\end{itemize}
\noindent Let $K$ be a simplicial complex model for the wedge sum of  $M(G_i,i)$, $1\leq i \leq m$. Then   $K$ satisfies $H_i(K)=G_i$ for $1\leq i\leq m$, and $K$ is a connected  simplicial complex whose dimension is less than or equal to $m+1$.   By Theorem~\ref{pr9.5.2.3}, there exists an acyclic  hypergraph $\mathcal{H}$ containing $K$ such that  $K_\mathcal{H}=\Delta[n]$ where $n$ is the dimension of $\mathcal{H}$,  $n$ is greater than the dimension of $K$ by $2$,   and   $H_i(\mathcal{H})\cong H_i(K)$ for any $i\geq 0$.  
With the help of Proposition~\ref{pr9.3.5},  $\mathcal{H}_0\cup\mathcal{H}_1$ is connected. The assertion follows. 
\end{proof}

\begin{remark}
By Proposition~\ref{ex9.5.2.1}, we see that if the abelian group $G_m$ is non-trivial, then the dimension of $\mathcal{H}$ in Theorem~\ref{th9.5.2.8} is greater than or equal to $m+2$.  Hence by Theorem~\ref{th9.5.2.8} (iii),  the dimension of $\mathcal{H}$ is either $m+2$ or $m+3$. 
\end{remark}

\section{Applications of the Embedded Homology in Data Analysis of Hyper-networks}\label{sec9.6}

Hypergraphs are   mathematical models of hyper-networks, which has significant applications in the data analysis of engineering, technology, economics, marketing, etc.     A hyper-network is a system consisting of players/items as well as relations among the  players/items.  For example, if we take all the google users in the world as players,  and assign a relation among the google users whenever they are in a   google group,  then  we get a hyper-network.    If we use a vertex  to represent a player/item  and use a hyperedge to represent a relation among the players/items, then we get a hypergraph model for a hyper-network.

In this section, by applying the embedded homology of hypergraphs, we construct the following indices:  a hyper-network connectivity index to measure the connectivity of the vertices of a hyper-network, a hyper-network differentiation index to measure the differentiation of the vertices of a hyper-network with respect to certain functions on the vertices, and a hyper-network correlation index to measure the correlation between two functions on the vertices of a hyper-network. This section is speculation-based and contains no mathematical results. Nevertheless,  the indices constructed in this section are possible to have potential applications in data analysis of hyper-networks. 

\smallskip

\subsection{The Hyper-network Connectivity Index}


Let $\mathcal{H}$ be a hypergraph. The connectivity of $\mathcal{H}$ measures how intimately the vertices of $\mathcal{H}$ are connected with each other by the hyperedges in $\mathcal{H}$. In recent years, the connectivity of $\mathcal{H}$ has been investigated from various aspects, for example,  \cite{c22,c3, c4,c1}.  Among these references, the connectivity is characterized from a homological aspect in \cite{c22}. We apply the $0$-th embedded homology of hypergraphs and give a {connectivity index} to measure the connectivity of $\mathcal{H}$.   

\smallskip

Let $\mathcal{H}^0=\mathcal{H}$. For any $k\geq 1$, by an induction on $k$, we define a sequence of operations (R$k$) and a sequence of hypergraphs $\mathcal{H}^k$ as follows:

\begin{quote}
(R$k$). For a vertex $v$ of $\mathcal{H}^{k-1}$, if there exist exactly $k$ hyperedges $\sigma_1,\cdots,\sigma_k$ in $\mathcal{H}^{k-1}$ such that $v$ is a vertex of each $\sigma_i$, $i=1,\cdots,k$, then we remove $v$ from each of $\sigma_1,\cdots,\sigma_k$. 
\end{quote}

\begin{quote}
($\mathcal{H}^k$). By taking the operation (R$k$) on the hypergraph $\mathcal{H}^{k-1}$ repeatedly until the hypergraph cannot be reduced by (R$k$) anymore, we obtain a hypergraph $\mathcal{H}^k$. 
\end{quote}

\noindent We define the {\bf hyper-network connectivity index} of  $\mathcal{H}$ to be the number
\begin{eqnarray*}
\text{Conn}(\mathcal{H})=\sum_{k\geq 0}\frac{\dim H_0(\mathcal{H}^k;\mathbb{Q})}{2^{k+1}\cdot |V_{\mathcal{H}^{k}}|},
\end{eqnarray*}
where $|\cdot|$ denotes the cardinality of a set. Then $\text{Conn}(\mathcal{H})$ is a positive number smaller than or equal to $1$, which reflects how intimately the vertices of $\mathcal{H}$ are connected by the hyperedges.  As $\text{Conn}(\mathcal{H})$ increases, the connectivity of the vertices of $\mathcal{H}$ becomes less significant.  In particular,  if $\text{Conn}(\mathcal{H})=1$, then $V_\mathcal{H}$ is totally discrete and there is no hyperedge  in $\mathcal{H}$ containing more than one point. 

\begin{example}\label{ex9.6.1.1}
Let each point   represents a person. Let  $G$ be the graph constructed by connecting two points  whenever the corresponding two persons are in parents-children relations or in spouse relations.  Then $\text{Conn}(G)$ can be used to measure how integrated the community is. If $\text{Conn}(G)$ becomes smaller/larger than before, then we conclude that the community becomes more/less integrated. 
\end{example}

\subsection{The Hyper-network Differentiation Index}

Let $\mathcal{H}$ be a hypergraph and let $\varphi: V_\mathcal{H}\longrightarrow [0,1]$ be a function on the vertices of $\mathcal{H}$. The problem how $V_\mathcal{H}$ is differentiated with respect to $\varphi$ and the relations given as hyperedges of $\mathcal{H}$ is investigated in \cite{d1} from a perspective of dynamical systems.  And the differentiation phenomenon of stocks in the financial market network is studied in \cite{c2}. We give a hyper-network differentiation index to measure how $V_\mathcal{H}$ is differentiated by using the embedded homology of hypergraphs.

\smallskip

For any $t\in [0,1]$, let $\mathcal{H}(t)$ be the hypergraph consisting of all the hyperedges of $\mathcal{H}$ whose vertices $v$  satisfy $\varphi(v)\geq t$.  That is, 
\begin{eqnarray*}
\mathcal{H}(t)=\Delta[v\in V_\mathcal{H}\mid \varphi(v)\geq t]\cap\mathcal{H}. 
\end{eqnarray*}
For any $i\geq 0$ and $n\geq 1$, the sequence of the dimension of the embedded homology  
\begin{eqnarray}\label{eq9.6.1.1}
\dim H_i(\mathcal{H}(k/n);\mathbb{Q}),\text{\ \ \ } k=0,1,\cdots,n
\end{eqnarray}
is a barcode (that is, a step function of one  variable which is a finite sum of constant functions on intervals). Letting $n\to \infty$, since $V_\mathcal{H}$ is finite, the barcode (\ref{eq9.6.1.1})  stabilises for $n$ sufficiently large.  We denote the limit barcode  as the following    function  
\begin{eqnarray*}
f_{i,\varphi}: [0,1]\longrightarrow \mathbb{R}_{\geq 0}. 
\end{eqnarray*}
Let $\gamma$ be a random variable in the function space
\begin{eqnarray*}
W_\varphi&=&\{\gamma: V_\mathcal{H}\longrightarrow [0,1]\mid \text{ for any } t\in [0,1], \text{ the number of vertices } v\in V_\mathcal{H} \\
&&\text{ such that } \varphi(v)\geq t \text{ equals to the number of vertices } v\in V_\mathcal{H} \text{ such that } \gamma(v)\geq t \}.
\end{eqnarray*}
We denote $E$ as the expectation and define $f_{i,\gamma}$ in the same way as $f_{i,\varphi}$ by only substituting $\varphi$ with $\gamma$.  Then  the degree of fitness
\begin{eqnarray*}
\text{Fit}(f_{i,\varphi}, E(f_{i,\gamma}\mid \gamma\in W_\varphi))
\end{eqnarray*}
 reflects the differentiation of $V_\mathcal{H}$ with respect to $\varphi$ and $\mathcal{H}$. Moreover, we let
\begin{eqnarray*}
\text{Diff}(\varphi, \mathcal{H})=\sum_{i\geq 0}  \frac{\text{Fit}(f_{i,\varphi}, E(f_{i,\gamma}\mid \gamma\in W_\varphi))}{2^{i+1}}. 
\end{eqnarray*}
We call $\text{Diff}(\varphi,\mathcal{H})$ the {\bf hyper-network differentiation index} of $\varphi$ with respect to $\mathcal{H}$. It is a number between $0$ and $1$.   As the number  $\text{Diff}(\varphi,\mathcal{H})$ increases,  the differentiation of  $V_\mathcal{H}$   with respect to $\varphi$ and  $\mathcal{H}$ becomes more significant.  

\begin{remark}
Given two non-negative  measurable functions $\beta_1,\beta_2$ on a measure space $X$ such that  
\begin{eqnarray*}
0<||\beta_1||_2,||\beta_2||_2<\infty  
\end{eqnarray*}
where $||\cdot||_2$  is the $L^2$-norm of a function,  we define the {\bf degree of fitness} between  $\beta_1$  and $\beta_2$ as 
\begin{eqnarray*}
\text{Fit}(\beta_1,\beta_2)=\frac{||\beta_1-\beta_2||_2}{||\beta_1||_2+ ||\beta_2||_2}. 
\end{eqnarray*}
The degree of fitness is a number between $0$ and $1$.  Smaller $\text{Fit}(\beta_1,\beta_2)$ means more significance 
of   the fitness   between  $\beta_1$  and  $\beta_2$.  
\end{remark}

\begin{example}\label{ex9.6.2.1}
Let  $G$ be the graph given in Example~\ref{ex9.6.1.1}. 
For any $n\geq 0$, whenever the points $v_0,\cdots, v_n$ of $G$ form a loop (a single point is regarded as a trivial loop), we give a hyperedge $\{v_0,\cdots,v_n\}$ in $\mathcal{H}$. 
Let $\varphi$ be a function with value in $[0,1]$ measuring the annual income/social status/personal property/education level of people. Then our  hyper-network differentiation index  $\text{Diff}(\varphi,\mathcal{H})$  can measure the social differentiation and mobility.  If $\text{Diff}(\varphi,\mathcal{H})$ becomes smaller/larger than  before, then we can conclude that the social differentiation decreases/increases. 
\end{example}

\subsection{The Hyper-network Correlation Index}

Correlation analysis  can be conducted on networks, for example, \cite{r1,r2}.  Let $\mathcal{H}$ be a hypergraph and let $\varphi, \psi: V_\mathcal{H}\longrightarrow [0,1]$ be two functions on the vertices of $\mathcal{H}$. To measure the correlation between $\varphi$ and $\psi$ with respect to the relations on $V_\mathcal{H}$ given as hyperedges of $\mathcal{H}$, we give a hyper-network correlation index  by using the embedded homology of hypergraphs.

\smallskip

For any $0\leq t,s\leq 1$, let $\mathcal{H}(t,s)$ be the hypergraph consisting of all the hyperedges of $\mathcal{H}$ whose  vertices $v$ satisfy $\varphi(v)\geq t$ and $\psi(v)\geq s$.  That is,
\begin{eqnarray*}
\mathcal{H}(t,s)=\Delta[v\in V_\mathcal{H}\mid \varphi(v)\geq t \text{ and } \psi(v)\geq s ]\cap\mathcal{H}. 
\end{eqnarray*}
For any $i\geq 0$ and $n\geq 1$, the sequence of the dimension of the embedded homology  
\begin{eqnarray}\label{eq9.6.2.1}
\dim H_i(\mathcal{H}(k/n, l/n);\mathbb{Q}),\text{\ \ \ } k,l=0,1,\cdots,n
\end{eqnarray}
is a $2$-dimensional barcode  (that is, a step function of two variables which is a finite sum of constant functions on squares).  Letting $n\to \infty$, since $V_\mathcal{H}$ is finite,   the $2$-dimensional barcode  (\ref{eq9.6.2.1})  stabilises when $n$ is sufficiently large.  We denote the limit barcode as the following function 
\begin{eqnarray*}
g_{i,\varphi,\psi}: [0,1]\times [0,1]\longrightarrow \mathbb{R}_{\geq 0}. 
\end{eqnarray*}
Let $\gamma_1\in W_\varphi$ and $\gamma_2\in W_\psi$ be independent random variables.  We define $g_{i,\gamma_1,\gamma_2}$ in the same way as $g_{i,\varphi,\psi}$ by only substituting $\varphi$ with $\gamma_1$ and substituting $\psi$ with $\gamma_2$.  Then the  degree of fitness  of $2$-variable functions
\begin{eqnarray*}
\text{Fit}(g_{i,\varphi,\psi}, E(g_{i,\gamma_1,\gamma_2}\mid \gamma_1\in W_\varphi,\gamma_2\in W_\psi))
\end{eqnarray*}
 reflects the correlation of $\varphi$ and $\psi$ on $V_\mathcal{H}$ with respect to  the hyperedges in $\mathcal{H}$. Moreover, we let
\begin{eqnarray*}
\text{Corr} (\varphi, \psi, \mathcal{H})=\sum_{i\geq 0}  \frac{\text{Fit}(g_{i,\varphi,\psi}, E(g_{i,\gamma_1,\gamma_2}\mid \gamma_1\in W_\varphi,\gamma_2\in W_\psi))}{2^{i+1}}. 
\end{eqnarray*}
We call $\text{Corr}(\varphi,\psi, \mathcal{H})$ the {\bf hyper-network correlation index} of $\varphi$ and $\psi$ with respect to $\mathcal{H}$.  It is a number between $0$ and $1$.  As the number $\text{Corr}(\varphi,\psi, \mathcal{H})$ increases, the correlation  between $\varphi$ and $\psi$ with respect to $\mathcal{H}$ becomes more significant.

\begin{example}\label{ex9.6.3.1}
Let 
$\mathcal{H}$ be the hypergraph given in Example~\ref{ex9.6.2.1}. 
Let $\varphi$ be a function with value in $[0,1]$ measuring the education level of a person. Let $\psi$ be a function with value in $[0,1]$ measuring the annual income of a person.  Then our  hyper-network correlation index  $\text{Corr}(\varphi,\psi, \mathcal{H})$  can measure the correlation between the education level and the annual income with the consideration of social relations. If $\text{Corr}(\varphi,\psi,\mathcal{H})$ becomes smaller/larger than  before, then we can conclude that education level becomes less/more related to annual income. 
\end{example}

\bigskip

{\noindent {\bf Acknowledgement}.  This work is supported by the National Research Foundation, Prime Minister's Office, Singapore under its Campus for Research Excellence and Technological Enterprise (CREATE) programme.  The authors would like to express their deep gratitude to the reviewer(s)  for their careful reading, valuable comments,  and helpful suggestions.
}

\vspace{0.88cm}

\vspace{0.88cm}

Stephane Bressan

School of Computing, National University of Singapore, Singapore, 117417. 

e-mail: steph@nus.edu.sg 
 
 \bigskip
 
 Jingyan Li

Department of Mathematics and Physics, Shijiazhuang Tiedao University, China, 050043. 

e-mail: yanjinglee@163.com

\bigskip


  \bigskip
 
 Shiquan Ren (for correspondence)
 
$^{\it a}$ School of Mathematics and Computer Science, Guangdong Ocean University, 1 Haida Road, Zhanjiang, China, 524088.
 
 $^{\it b}$ Department of Mathematics,   National University of Singapore, Singapore, 119076.

e-mail: sren@u.nus.edu 

\bigskip

Jie Wu

  Department of Mathematics,   National University of Singapore, Singapore, 119076.

e-mail: matwuj@nus.edu.sg






\end{document}